\documentclass[11pt]{article}
\usepackage[utf8]{inputenc}
\usepackage{amsthm}
\usepackage{amsmath}
\usepackage{amssymb}
\usepackage{tikz}
\usepackage{geometry}
\usepackage{authblk}
\usepackage{hyperref}

\newtheorem{thm}{Theorem}[section]
\newtheorem{dfn}{Definition}[section]
\newtheorem{prop}{Proposition}[section]
\newtheorem{cor}{Corollary}[section]
\newtheorem{lem}{Lemma}[section]

\newtheorem{theorem}{Theorem}

\newcommand\V{V_{\mathcal{M}}}
\newcommand\M{\mathcal{M}}
\newcommand\N{\mathcal{N}}
\newcommand\K{K_{\mathcal{M}}}
\newcommand\teich{\tilde{\mathcal{H}}(\kappa)}
\newcommand\FM{\mathcal{F}^{\mathcal{M}}}
\newcommand\moduli{\mathcal{H}(\kappa)}
\newcommand\GL{GL_2^{+}(\mathbb{R})}
\newcommand\SL{SL_2(\mathbb{R})}

\newcommand{\R}{\mathbb{R}}
\newcommand{\Q}{\mathbb{Q}}
\newcommand{\Z}{\mathbb{Z}}
\newcommand{\frakK}{\mathfrak{K}_{\M}}

\title{A criterion for density of the isoperiodic leaves in rank one affine invariant orbifolds}
\author{Florent Ygouf \\ \texttt{florentygouf@mail.tau.ac.il}}
\affil{School of Mathematical Sciences, Tel Aviv University}

\begin{document}

\parindent=0em

\maketitle

\begin{abstract}
We define on any affine invariant orbifold $\M$ a foliation $\FM$ that generalises the isoperiodic foliation on strata of the moduli space of translation surfaces and study the dynamics of its leaves in the rank 1 case. We establish a criterion that ensures the density of the leaves and  provide two applications of this criterion. The first one is a classification of the dynamical behavior of the leaves of $\FM$ when $\M$ is a connected component of a Prym eigenform locus in genus 2 or 3 and the second provides the first examples of dense isoperiodic leaves in the stratum $\mathcal{H}(2,1,1)$. 
\end{abstract}

\section{Introduction}\label{introduction}

\subsection{Context}

Let $g \geq 1$, $n \geq 1$ and $\kappa = (k_1,\cdots,k_n)$ be a integer partition of $2g-2$. We denote by $\mathcal{H}(\kappa)$ the space of isomorphism classes of triples $q = (X_q,\mathfrak{z}_q,\omega_q)$ where $X_q$ is a genus $g$ Riemann surface, $\omega_q$ is a non-vanishing holomorphic $1$-form on $X_q$ and $\mathfrak{z}_q$ is a bijection from $\{1,\cdots,n\}$ to $\Sigma_q$ where $\Sigma_q$ is the set of zeroes of  $\omega_q$ and such that $\mathfrak{z}_q(i)$ has multiplicity $\kappa_i$. Such a $q$ is usually referred to as a pointed translation surface or a translation surface with zeroes labeled. Occasionally, we might also say surface for the sake of simplicity. The space $\moduli$ admits a natural action of the group $\GL$, which is a generalization of the action by left multiplication of $\GL$ on the space of flat tori $\GL / SL_2(\mathbb{Z})$. The classification of the closed invariant sets is a central problem in Teichmüller dynamics. Recently, the work of Eskin, Mirzakhani and Mohammadi has shed light on the structure of such sets: they are immersed orbifolds cut out by linear equations with real coefficients in period coordinates. These objects are referred to as affine invariant orbifolds. See \cite{eskin2015isolation}. 

\medskip

Transverse to the $\GL$-action, there is a local action by $\mathbb{C}^{n-1}$ that fits into a holomorphic foliation of $\moduli$. It is usually referred to as the isoperiodic foliation, the absolute period foliation, the kernel foliation or the Rel foliation. The leaf $\mathcal{F}_q$ of a translation surface $q$ is locally obtained as a level set of the absolute periods map. See \cite{hooper2015rel} for a careful definition of this foliation. This foliation has been introduced about 25 years ago by Kontsevich and Eskin, and later by McMullen and Calta before it became a central object in Teichmüller dynamics. See for example how it is involved in the classification of horocycle orbits of Prym eigenforms in $\mathcal{H}(1,1)$ obtained by Bainbridge, Smillie and Weiss in \cite{bainbridge2016horocycle}. 

\medskip 

Several papers have been devoted to understanding the dynamics of its leaves. McMullen showed that the foliation is ergodic with respect to the Masur-Veech measure in the principal stratum $\mathcal{H}(1,\cdots,1)$ in genus 2 and 3 using Ratner's theory in \cite{mcmullen2014moduli}. Hooper and Weiss gave the first examples of dense leaves outside the principal stratum in \cite{hooper2015rel}. Shortly after, Calsamiglia, Deroin and Francaviglia have obtained a Ratner-like classification of the minimal sets in the principal stratum and obtained the ergodicity with respect to the Masur-Veech measure as a consequence of this classification. See \cite{calsamiglia2015transfer}. Simultaneously, Hamenstädt gave another proof of the ergodicity in the principal strata in \cite{hamenstadt2018ergodicity}. Both these last two results used McMullen's result as a base case for an induction. Surprisingly, apart from the examples of Hooper and Weiss, nothing is known for the dynamics of the isoperiodic foliation in strata where at least one zero is not simple. 

\medskip

\subsection{Statement of the results}

We define on any affine invariant orbifold $\M$ a foliation $\FM$ that generalizes the isoperiodic foliation on the strata of the moduli space. These foliations are particularly interesting in the light of Eskin, Mirzakhani and Mohammadi's result as any information on the geometry of affine invariant orbifolds directly translates into information on the dynamics of the $\GL$-action. The dynamics of the foliation $\FM$ also contains information on the dynamics of the isoperiodic foliation itself as any leaf  $\FM_q$ is a connected component of $\mathcal{F}_q \cap \M$. 

\medskip 

The foliation $\FM$ will be referred to as the $\M$-isoperiodic foliation. For some affine invariant orbifolds, the leaves of $\FM$ have dimension $0$. It is for instance the case for Teichmüller discs or for hyperelliptic loci of non-hyperelliptic strata. More sophisticated examples are given in \cite{eskin2018billiards}. The other affine invariant orbifolds are called nonabsolute. In this text we study the dynamics of the foliation $\FM$ in the case where $\M$ is a nonabsolute rank 1 affine invariant orbifold. We consider a property $\mathcal{P}$ that describes cylinder decompositions of a certain type in $\M$ and we establish the following criterion: 

\begin{theorem}\label{A}
Let $\M$ be a nonabsolute rank 1 affine invariant orbifold. If $\M$ satisfies property $\mathcal{P}$, then all the leaves of $\FM$ are projectively dense. 
\end{theorem}

Here, projectively dense means that the leaf of any surface $q$ is dense in the locus of surfaces in $\M$ that have the same area as $q$. Equivalently, the projection of $\FM_q$ in $\mathbb{P}\M$ is dense. This has to do with the fact that the area of a translation surface can be computed only with its absolute periods and thus the area function is constant along the leaves of the $\M$-isoperiodic foliation. See Definition \ref{projectivedensity} for more details. The proof of Theorem \ref{A} relies on the fact that for a surface $q$ with a cylinder decomposition provided by property $\mathcal{P}$, certain  deformations tangent to $\FM$ accumulate on the horocycle orbit of $q$. These deformations are usually referred to as the Real Rel flows, see for instance \cite{hooper2015rel}. From this, we deduce that the horocycle orbit is actually contained in the closure of the leaf of $q$ and with a bit more work using the geodesic flow, we show that the whole $\SL$-orbit closure is, concluding the proof. We then give examples of affine invariant orbifolds that have property $\mathcal{P}$ and others that do not. Special attention is dedicated to Prym eigenforms of genus 2 and 3 and we show the following: 

\begin{theorem}\label{B}
Let $\kappa = (1,1), \ (2,2)^{odd}, (2,1,1)$ or $(1,1,1,1)$ and let $\M$ be a connected component of $\Omega E_D(\kappa)$. Then either all the leaves of $\FM$ are closed or all the leaves of $\FM$ are projectively dense. The last case occurs if, and only if, $D$ is not a square. 
\end{theorem}

The space $\Omega E_D(\kappa)$ appearing in the statement is the Prym eigenform locus, discovered by McMullen, see \cite{mcmullen2007Prym}. The remaining connected components $\mathcal{H}(2)$, $\mathcal{H}^{odd}(4)$, $\mathcal{H}^{hyp}(4)$ and $\mathcal{H}(2,2)^{hyp}$ in genus 2 or 3 are not included in Theorem \ref{B} as the corresponding Prym eigenform loci are union of Teichmüller discs and the leaves have dimension 0, hence closed even when D is not a square. This is clear when there is only a single singularity. See Proposition 2.3 \cite{lanneau2014connected} for a proof of that claim in $\mathcal{H}^{hyp}(2,2)$. Finally, we use Theorem \ref{A} to compute the closure of some isoperiodic leaves:  

\begin{theorem}\label{C}
Let $q \in \Omega E_D(2,1,1)$ where $D$ is not a square. The leaf $\mathcal{F}_q$ of $q$ is projectively dense in $\mathcal{H}(2,1,1)$. 
\end{theorem}

We emphasize that in Theorem \ref{B}, we consider the isoperiodic foliation inside $\Omega E_D$ while in theorem \ref{C} we consider the whole isoperiodic foliation. If $\M$ denotes the connected component of $\Omega E_D(2,1,1)$ that contains $q$ then $\FM_q$ has complex codimension 1 in $\mathcal{F}_q$. Theorem \ref{C} gives, to our knowledge, the first examples of dense isoperiodic leaves in $\mathcal{H}(2,1,1)$.

\subsection{Organisation of the paper}

In Section \ref{framework}, we collect relevant definitions. We review the structures of the strata and of the closed $\GL$-invariant sets, recall the definitions of the Prym eigenform loci, and define the foliation $\FM$. In Section \ref{twists}, we introduce the twist map and define the spaces $\K$ and $\V$. In Section \ref{p}, we establish the criterion presented in Theorem \ref{A}. In Section \ref{examples}, we provide a list of affine invariant orbifolds that have property $\mathcal{P}$, and deduce Theorem \ref{B}. Finally, Section \ref{full} is dedicated to the proof of Theorem \ref{C}.

\subsection{Acknowledgements}

I am deeply indebted to Erwan Lanneau for introducing to the world of translation surfaces and suggesting the questions studied in this text, and many more. I would also like to thank Barak Weiss, Pascal Hubert, Duc-Manh Nguyen and Alex Wright for having discussed with me these results which have significantly benefited from their insightful comments. This work has been partially supported by the LabEx PERSYVAL-Lab (ANR-11-LABX-0025-01) funded by the French program Investissement d’avenir. 

\section{Framework}\label{framework}

\subsection{Period coordinates \& the action of $\GL$}

Let $S$ be a genus $g$ surface and let $\Sigma = \{z_1,\cdots,z_n\} \subset S$ be a finite set. We denote by $\tilde{\mathcal{H}}(\kappa)$ the space of equivalence classes of marked pointed translation surfaces $(q,f)$ where $f: S \to X_q$ is a homeomorphism such that $f(z_i) = \mathfrak{z}_q(i)$. Two marked pointed translation surfaces $(q_1,f_1)$ and $(q_2,f_2)$ are equivalent if there is a biholomorphism $\varphi : X_{q_1} \to X_{q_2}$ such that $\varphi^{\ast} \omega_{q_2} = \omega_{q_1}$, $\varphi \circ \mathfrak{z}_{q_1} = \mathfrak{z}_{q_2}$ and $f_2^{-1} \circ \varphi \circ f_1$ is homotopic to the identity of $S$ rel $\Sigma$. The following map is known as the period map : 

$$
\Phi : 
\begin{matrix}
\tilde{\mathcal{H}}(\kappa) &\to & H^1(S,\Sigma,\mathbb{C}) \\
(q,f) &\mapsto & (\gamma \mapsto \int_{f\circ \gamma} \omega_q)
\end{matrix}
$$

There is a complex structure on $\tilde{\mathcal{H}}(\kappa)$ that turns $\Phi$ into a local biholomorphism, and if $MCG(S,\Sigma)$ denotes the relative mapping class group of $S$ that fixes $\Sigma$ point-wise (\textit{i.e.} $\forall h \in MCG(S,\Sigma), \ h(z_i) = z_i$), then there is a right action of $MCG(S,\Sigma)$ on $\teich$ by precomposition: $(q,f) \cdot h = (q, f \circ h)$. The quotient set $\teich / MCG(S,\Sigma)$ is in bijection with $\moduli$ and the latter is endowed with the complex orbifold structure that turns the canonical projection $\pi : \tilde{\mathcal{H}}(\kappa) \to \mathcal{H}(\kappa)$ into a local biholomorphism (in the orbifold sense, see \cite{caramello2019introduction} for relevant definitions). The space $\tilde{\mathcal{H}}(\kappa)$ is endowed with a left $GL_2^{+}(\mathbb{R})$-action defined by : 
$$
\forall g \in \GL \ \Phi(g\cdot (q,f)) = g \cdot \Phi (q,f)
$$

where $\GL$ acts on $H^1(S,\Sigma,\mathbb{C})$ by post-composition. That action descends to an action on $\moduli$ in a way that the canonical projection $\pi$ is $\GL$-equivariant. We will use the following notations:

$$ 
g_t = \begin{pmatrix} e^t && 0 \\ 0  && e^{-t} \end{pmatrix}, \ \ h_t = \begin{pmatrix} 1 && t \\ 0  && 1 \end{pmatrix}
$$

The actions of $A=\{g_t \ | \ t \in \R \}$ and $U = \{h_t \ | \ t \in \R \}$ are known as the geodesic flow and horocycle flow. More details on the structure of these spaces and the action of $\GL$ can be found in \cite{zorich2006flat} or \cite{forni2013introduction}. 

\begin{dfn}[Affine invariant orbifold]\label{affinemanifold}
An affine invariant orbifold is a closed connected subset $\M$ of $\moduli$ obtained as the image of a complex orbifold $\mathfrak{M}$ by a proper immersion $\iota$ that satisfies the following property: if $\mathfrak{q} \in \mathfrak{M}$, there is an open set $\mathcal{U}$ around $\mathfrak{q}$, an orbifold chart $(\varphi,\mathcal{V},\Gamma)$ around $\iota(\mathfrak{q})$ and a $\mathbb{R}$-linear subspace $V$ of $H^1(S,\Sigma,\mathbb{R})$ such that: 

\begin{equation}\label{localmodel}
\iota(\mathcal{U})  = \varphi(\mathcal{V} \cap V \otimes \mathbb{C})
\end{equation}
\end{dfn}

We recall that an orbifold chart $(\varphi,\mathcal{V},\Gamma)$ around $q \in \moduli$ is the data of an open subset $\mathcal{V} \subset H^1(S,\Sigma,\mathbb{C})$ and a finite subgroup $\Gamma$ of $MCG(S,\Sigma)$ acting linearly on $H^1(S,\Sigma,\mathbb{C})$ together with a continuous $\Gamma$-invariant map $\varphi: \mathcal{V} \to \moduli$ such that the induced map $\varphi:\mathcal{V}/\Gamma \to \varphi(\mathcal{V})$ is a homeomorphism. We also recall that by definition of the orbifold structure on $\moduli$ we have $\Phi \circ \varphi = \pi$ whenever this is well defined. The term immersion refers to an immersion of orbifold, See \cite{caramello2019introduction} for relevant definitions. Affine invariant orbifolds are invariant under the action of $\GL$ and Eskin, Mirzakhani and Mohammadi proved in a celebrated result that any $\GL$-orbit closure is of that form. See \cite{eskin2015isolation}. We will say that a subspace $V$ as in the equation \eqref{localmodel} is a local model of $\M$ at $q$. An important numerical invariant associated to these affine invariant orbifolds is the rank, defined as follows: define $\rho :  H^1(S,\Sigma,\mathbb{C}) \to H^1(S,\mathbb{C})$ to be the canonical restriction map, and chose a local model $V$ of $\M$. Avila, Eskin and Möller proved in \cite{avila2017symplectic} that $\rho(V \otimes \mathbb{C})$ is a symplectic subspace of $H^1(S,\mathbb{C})$ (with respect to the intersection form). The rank of $\M$, denoted $rk(\M)$, is then defined as half the complex dimension of this space. More details can be found in \cite{wright2015cylinder}. The following definition will be important for the remainder of this text:

\begin{dfn}
Let $\M$ be a affine invariant orbifold. The field of definition of $\M$ is the smallest subfield $k(\M)$ of $\mathbb{R}$ such that any $V$ as in \eqref{localmodel} can be written as $V = V_0 \otimes_{k(\M)} \mathbb{R}$, where $V_0$ is a $k(\M)$-linear subspace of $H^1(S,\Sigma,k(\M))$
\end{dfn}

By definition, affine invariant orbifolds are cut out by real linear equations and the field of definition of $\M$ is thus the smallest subfield of $\mathbb{R}$ in which those linear equations take their coefficients. We shall say that an affine invariant orbifold is arithmetic if $k(\M)=\mathbb{Q}$ and nonarithmetic otherwise. Wright proved that the field of definition can be computed using periodic translation surfaces: let $q$ be a translation surface. The bilinear form $\omega_q \otimes \overline{\omega_q}$ induces a metric on $X_q$ that is flat away from the singularities of $\omega$. We say that $q$ is periodic in direction $\theta$ if the geodesic rays in direction $\theta$ are either periodic, or saddle connections (paths that start and end at singularities of $\omega_q$). It is well known that if a translation surface is periodic in a given direction, it is decomposed as a finite union of cylinders, which are subsets isometric to $[0,h] \times \mathbb{R} / c \mathbb{Z}$ and whose boundary components are union of saddle connections in direction $\theta$. Wright proved the following in \cite{wright2015cylinder}:

\begin{prop}\label{field}
Let $\M$ be an affine invariant orbifold and let $q$ be a periodic surface in $\M$ with $m$ cylinders. Then $k(\M)$ is contained in $\mathbb{Q}[c_2c_1^{-1},\cdots,c_mc_1^{-1}]$ where the $c_i$ are the circumferences of the cylinders of $q$.
\end{prop}

Wright actually proved a stronger result and proved the other inclusion, provided one considers only a subclass of cylinders. See \cite{wright2015cylinder} for more details. 

\subsection{Prym eigenform loci}

In this section, we recall a construction giving an infinite family of rank one affine invariant orbifolds discovered by McMullen. Let $q$ be a translation surface endowed with a holomorphic involution $\lambda$. We denote by $\Omega(X_q)$ the set of holomorphic $1$-forms on $X_q$, and by $\Omega^-(X_q)$ the sub-space of $\lambda$-anti-invariant holomorphic $1$-forms. We say that $q$ is a Prym form if $\omega_q \in \Omega^-(X_q)$, that is $\lambda^{\ast} \omega_q = - \omega_q$, and $dim \ \Omega^-(X_q) = 2$. The Prym variety $Prym(q,\lambda)$ is defined as the 2-dimensional abelian variety $(\Omega^-(X_q))^{\ast} / H_1^-(X_q,\mathbb{Z})$ endowed with the polarization coming from the intersection form on $H_1(X_q,\Z)$. Finally, let $D$ be a positive integer congruent to 0 or 1 mod 4 and let $\mathcal{O}_D \simeq \mathbb{Z}[X]/(X^2+bX+c)$ be the real quadratic order of discriminant $D = b^2-4c$. 

\begin{dfn}
A Prym eigenform is a Prym form $q$ corresponding to an involution $\lambda$ such that there is an injective ring morphism $\mathfrak{i} : \mathcal{O}_D \to End(Prym(q,\lambda))$ satisfying the following properties:

\begin{enumerate}
\item $\mathfrak{i}(\mathcal{O}_D)$ is a proper subring comprised of self-adjoint endomorphisms
\item $\omega_q$ is an eigenvector for the action of $\mathcal{O}_D$ on $\Omega(X_q)^{-}$
\end{enumerate}
\end{dfn}

We will denote by $\Omega E_D(\kappa)$ the set Prym eigenforms contained in the stratum $\moduli$. For more details, see \cite{mcmullen2007Prym}. McMullen proved the following:

\begin{prop}
Let $\M$ be a connected component of $\Omega E_D(\kappa)$. Then $\M$ is a rank one affine invariant orbifold whose field of definition is $\Q(\sqrt{D})$.  
\end{prop}

The Prym eigenforms play a crucial role in McMullen's classification of affine invariant orbifolds in genus $2$. It is proved in \cite{mcmullen2007dynamics} that if the orbit of a surface is neither closed nor dense in the stratum in which it belongs, then it is a Prym eigenform, the Prym involution being given by the hyperelliptic involution. In \cite{mcmullen2007Prym}, infinite families of Prym eigenforms are constructed in genus up to $5$, and it is a consequence of Riemann-Hurwitz that Prym eigenforms cannot exist in genus bigger than 5.  

\subsection{The $\M$-isoperiodic foliation}

The aim of this subsection is to define a foliation on any affine invariant orbifold. In that perspective, let us first recall the definition of the usual isoperiodic foliation $\mathcal{F}$: the map $\rho \circ \Phi$ is a $MCG(S,\Sigma)$-equivariant submersion. The isoperiodic foliation is defined as the quotient foliation of the foliation of $\teich$ by connected components of the level sets of $\rho \circ \Phi$. Equivalently, for any orbifold chart $(\varphi,\mathcal{U},\Gamma)$ around $q$ with $\varphi(\xi)=q$, we have $\varphi(\mathcal{U} \cap \xi + \mathrm{ker}\rho) = \varphi(\mathcal{U}) \cap \mathcal{F}_q$. We are going to adapt this construction. Let $\M$ be an affine invariant orbifold and $q$ be a surface in $\M$. We denote by $\FM_q$ the connected component of $\mathcal{F}_q \cap \M$ that contains $q$.

\begin{prop}\label{foliation} Let $\iota : \mathfrak{M} \to \moduli$ be an immersion such that $\iota(\mathfrak{M}) = \M$ as in definition \ref{affinemanifold}. There is a foliation $\mathfrak{F}$ on $\mathfrak{M}$ such that for any $q \in \M$ and $\mathfrak{q} \in \mathfrak{M}$ such that $\iota(\mathfrak{q}) = q$, there is a neighborhood $\mathcal{U}$ around $\mathfrak{q}$ in $\mathfrak{M}$ such that:

$$
\iota(\mathcal{U}\cap \mathfrak{F}_{\mathfrak{q}}) = \iota(\mathcal{U})\cap \FM_q
$$

\end{prop}

\begin{proof}
Let $\mathfrak{q} \in \mathfrak{M}$ and let $(\varphi_1,\mathcal{U}_1,\Gamma_1)$ be an orbifold chart around $\mathfrak{q}$. We recall that this means that $\varphi_1:\mathcal{U}_1 \to \mathfrak{M}$ is a continuous $\Gamma_1$-equivariant map and that the induced map $\varphi_1:\mathcal{U}_1/\Gamma_1 \to \varphi_1(\mathcal{U}_1)$ is a homeomorphism that contains $\mathfrak{q}$ in its image. Let also $(\varphi_2, \mathcal{U}_2,\Gamma_2)$ be a chart around $q = \iota(\mathfrak{q})$ and, if $\mathcal{U}_1$ is small enough, there is a $V \in H^1(S,\Sigma,\mathbb{C})$ such that $\iota\circ \varphi_1 (\mathcal{U}_1) = \varphi_2(\mathcal{U}_2 \cap V \otimes \mathbb{C})$, as in the Definition \ref{affinemanifold}. By definition of being an immersion, up to taking even smaller $\mathcal{U}_1$ and $\mathcal{U}_2$, there is a smooth local lift $\tilde{\iota}:\mathcal{U}_1 \to \mathcal{U}_2$ such that $\iota \circ \varphi_1 = \varphi_2 \circ \tilde{\iota}$ and the differential of $\tilde{\iota}$ at any point of $\mathcal{U}$ is injective. Up to shrinking again the open sets and replacing $\tilde{\iota}$ by $\gamma \cdot \tilde{\iota}$ for some $\gamma \in \Gamma_2$ if necessary, we can assume that the image of $\tilde{\iota}$ is contained in $\mathcal{U}_2 \cap V\otimes \mathbb{C}$. Up to shrinking one last time $\mathcal{U}_1$ and $\mathcal{U}_2$, we can assume by the Local Inversion Theorem that $\tilde{\iota}$ is actually a diffeomorphism from $\mathcal{U}_1$ to $\mathcal{U}_2 \cap V\otimes \mathbb{C}$. We thus get a foliation $\mathfrak{F}_{\mathcal{U}_1}$ of $\mathcal{U}_1$ by level sets of $\rho \circ \tilde{\iota}$ whose leaves are permuted by the action of $\Gamma_1$. Indeed, for any $\gamma_1 \in \Gamma_1$ there is a $\gamma_2 \in \Gamma_2$ such that $\tilde{\iota} \circ \gamma_1 = \gamma_2 \circ \tilde{\iota}$ and $\gamma_2$ acts trivially on $\mathrm{ker}\rho$. Let $\hat{\mathfrak{q}} \in \mathcal{U}_1$, $\mathcal{U} = \varphi_1(\mathcal{U}_1)$ and compute:  

\begin{align*}
    \iota \circ \varphi_1(\mathfrak{F}_{\mathcal{U}_1,\hat{\mathfrak{q}}}) &= \varphi_2 \circ \tilde{\iota}(\mathfrak{F}_{\mathcal{U}_1,\hat{\mathfrak{q}}}) \\
    &= \varphi_2(\mathcal{U}_2 \cap V \otimes \mathbb{C} \cap \mathrm{ker} \rho + \tilde{\iota}(\hat{\mathfrak{q}})) \\
    &= \varphi_2(\mathcal{U}_2 \cap V \otimes \mathbb{C}) \cap \varphi_2(\mathcal{U}_2 \cap \mathrm{ker} \rho + \tilde{\iota}(\hat{\mathfrak{q}})) \\
    &= \iota\circ \varphi_1 (\mathcal{U}_1) \cap \varphi_2(\mathcal{U}_2) \cap \mathcal{F}_q = \iota(\mathcal{U}) \cap \FM_q 
\end{align*}

To sum up the construction we just made, we found an orbifold chart $(\mathcal{U},\varphi,\Gamma)$ around any $\mathfrak{q} \in \mathfrak{M}$ such that $\mathcal{U}$ is endowed with a $\Gamma$-invariant foliation $\mathfrak{F}_{\mathcal{U}}$ such that for any $\hat{\mathfrak{q}} \in \mathcal{U}$, we have $\iota \circ \varphi (\mathfrak{F}_{\mathcal{U},\hat{q}}) = \iota \circ \varphi(\mathcal{U}) \cap \FM_q$ where $q = \iota \circ \varphi(\hat{q})$. It is easily verified that these foliated charts are compatible, this is essentially due to the fact that two different local lifts of $\iota$ differ by an element of $MCG(S,\Sigma)$, which acts trivially on $\mathrm{ker}\rho$. We denote by $\mathfrak{F}$ the induced foliation of $\mathfrak{M}$. By construction, the leaves of this foliation satisfy the equation of the statement. 

\end{proof}

Given the nature of Proposition \ref{foliation}, we shall say that the $\FM_q$ are the leaves of an immersed foliation, which will be referred to as the \textbf{$\M$-isoperiodic foliation} and will denote by $\FM$. 

\medskip 

\textbf{Examples.} 

\begin{enumerate}

\item If $\M=\moduli$, the foliation $\FM$ is the usual isoperiodic foliation (or Kernel foliation or Rel foliation).
\item Let $\M$ and $\N$ be two affine invariant orbifolds with $\M \subset \N$. Then the $\M$-isoperiodic leaf of any surface of $\M$ is contained in its $\N$-isoperiodic leaf. 
\item Let $q$ be a flat cover of a Veech surface $q_0$ that ramifies over the singularities of $\omega_{q_0}$ and a non periodic point. Let $\M$ be the $\GL$-orbit closure of $q$. Then the $\M$-isoperiodic leaf of any surface in $\M$ is a finite flat cover of the surface $q_0$ punctured at the singularities of $\omega_{q_0}$.
\item Let $\M$ be a Teichmuller disc. Then the $\M$-isoperiodic leaf of any surface $q \in \M$ is the singleton $\{q\}$.     
\item Let $\M = \Omega E_D(\kappa)$. Let $q \in \M$ and let $\lambda$ be the Prym involution of  $q$. Then $\FM_q$ is modeled on $\mathrm{ker}(\lambda^{\ast}+id) \cap \mathrm{ker}(\rho) \subset  H^1(S,\Sigma,\mathbb{C})$. See Proposition \ref{localmodelleafforeigen}.
\end{enumerate}

Notice that if $V$ is a local model of $\M$, then the space $V \cap \mathrm{ker}\rho$ does not depend on $V$. This is due to the fact that $MCG(S,\Sigma)$ preserves $\Sigma$ point-wise and thus acts trivially on $\mathrm{ker}\rho$ together with the fact that $\M$ is connected. We denote this space by $\frakK$. The complex linear space $\frakK \otimes \mathbb{C}$ is canonically identified with the tangent space of the leaves of $\FM$ in period coordinates. We can also define $\rho_q : H^1(X_q,\Sigma_q,\mathbb{C}) \to H^1(X_q,\mathbb{C})$. Since $MCG(S,\Sigma)$ does not permute $\Sigma$, the space $\mathrm{ker}\rho$ is canonically isomorphic to $\mathrm{ker}\rho_q$ via any choice of marking on $q$. We will use thus identify these spaces from now on. 

\begin{dfn}
An affine invariant orbifold $\M$ is said to be nonabsolute when leaves of $\FM$ have positive dimension. This is equivalent to the fact $\frakK$ has positive dimension. 
\end{dfn}

By definition of the rank, an affine invariant orbifold $\M$ is nonabsolute whenever $dim_{\mathbb{C}}(\M) > 2rk(\M)$. As an example, we have the following: 

\begin{prop}\label{localmodelleafforeigen}
Let $\M$ be a connected component of $\Omega E_D(\kappa)$, let $q \in \M$ and let $\lambda$ be the prym involution on $q$, then $\frakK$ is equal to $\mathrm{ker}\rho_q \cap \mathrm{ker}(\lambda^{\ast} + id)$. 
\end{prop}

We omit the proof of Proposition \ref{localmodelleafforeigen} as it folklore and a bit technical though easily understandable: it simply means that for an isoperiodic deformation of $q$ to remain in $\M$, it is enough that the Prym involution is preserved along the deformation. This is due to the fact that being an eigenform for $\mathcal{O}_D$ is a condition that only involves the absolute periods. Write $| \kappa | = n^+ + n^-$ where $n^+$ is the number of zeroes fixed by the Prym involution that defines $\M$. In particular $n^-$ is even and we claim that $\frakK$ has dimension $\frac{n^-}{2}$. Indeed, let $q \in \M$, let $\lambda$ be the corresponding Prym involution and let $z \in \Sigma_q$. Let $\gamma_z$ be a homotopically trivial loop that goes around $z$ clockwise and let $v_{z}= I(\gamma_z - \lambda_{\ast} \gamma_z,\cdot) \in \mathrm{ker}\rho_q$, where $I: H_1(X_q - \Sigma_q,\mathbb{Z}) \times H_1(X_q,\Sigma_q,\mathbb{Z}) \to \mathbb{Z}$ is the intersection product. The fact that $v_z \in \mathrm{ker}\rho_q$ comes from the fact that $\gamma_z$ is homotopically trivial. Notice that $v_z = 0$ if, and only if, the point $z$ is fixed by $\lambda$, that $\lambda^{\ast} v_z = -v_z$ and that the $v_z$ are a generating family of $\mathrm{ker}\rho_q \cap \mathrm{ker}(\lambda^{\ast} + id)$ whose span has dimension $\frac{n^-}{2}$. We conclude using Proposition \ref{localmodelleafforeigen}. 

\medskip 

The $\M$-isoperiodic foliation behaves nicely with respect to the action of $\GL$, as shows the following:

\begin{prop}\label{product}
Let $\mathcal{M}$ be a nonabsolute rank one affine invariant orbifold, let $q$ be a translation surface in $\mathcal{M}$ and let $g \in \GL$. The following formulas holds : 

\begin{equation}\label{commutation}
g \cdot \mathcal{F}^{\mathcal{M}}_q = \mathcal{F}^{\mathcal{M}}_{g\cdot q}
\end{equation} 

and 

\begin{equation}\label{product2}
\GL \cdot \FM_{q} = \M
\end{equation}
\end{prop}

\begin{proof}
Equation \eqref{commutation} is a consequence of the fact the Period map is $\GL$-equivariant. For equation \eqref{product2}, let $q_1$ be a translation surface in $\M$. Notice that the sets $\GL \cdot \FM_{q}$ and $\GL \cdot \FM_{q_1}$ both contain an open set and since the action of $\GL$ on $\M$ is topologically transitive, we deduce that there is a matrix $g \in \GL$ such that $g \cdot \FM_q \cap \FM_{q_1}$ is non empty. But by \eqref{commutation} $g \cdot \FM_q = \FM_{g \cdot q}$ and thus $g \cdot \FM_q = \FM_{q_1}$. This concludes. 
\end{proof}

Recall that the area of a surface $q \in \moduli$ is defined as the integral of $\frac{i}{2} \cdot \omega_q \wedge \overline{\omega_q}$ over $q$. It is well known that the area of $q$ only depends upon the integral of $\omega$ over the absolute cycles, and thus the area function is constant along the leaves of the $\M$-isoperiodic foliation. This leads to the following definition. 

\begin{dfn}\label{projectivedensity}
Let $\M$ be a nonabsolute affine invariant orbifold and let $q$ be a surface in $\M$. The leaf $\FM_q$ is said to be projectively dense if it is dense in the set of surfaces in $\M$ that have the same area as $q$.
\end{dfn}

We end this section by recalling that there is a family of flows, called rel flows, on $\moduli$ that are tangent to the isoperiodic foliation. They are obtained as follows: any $v \in \mathrm{ker}\rho$ induces a constant vector field on $H^1(S,\Sigma,\mathbb{C})$, which can be pulled back to a vector field on $\teich$ by the period map $\Phi$. Since the action of $MCG(S,\Sigma)$ preserves point-wise $\Sigma$, this vector field descends to a well-defined vector field on $\moduli$. We denote by $Rel_v^t$ the corresponding flow. See Section 4 \cite{bainbridge2016horocycle} for more details on rel flows and the isoperiodic foliation (called Rel foliation in the reference). If $v \in \frakK \otimes \mathbb{C}$, the flow $Rel_v^t$ preserves $\M$ by construction. If $v \in \frakK$, we say that $Rel_v^t$ is a real rel flow. 

\section{Modifying the twist parameters}\label{twists}

Let $q \in \moduli$ be a horizontally periodic surface. The purpose of this section is to analyze some deformations of the surface $q$ that preserve its cylinder decomposition. This will allow us to have a better understanding of its horocycle and real rel trajectories and will lead to the formulation of our property $\mathcal{P}$.

\subsection{The twist map}

Suppose $q$ is decomposed into $m$ horizontal cylinders, which we denote by $\mathcal{C}_1, \cdots, \mathcal{C}_m$. We denote by $\gamma_i$ the core curve of the cylinder $\mathcal{C}_i$. One can define a smooth map (in the orbifold sense) $\tau$ from $\mathbb{R}^m$ to $\moduli$ in the following way: choose $(x_1,\cdots,x_m)$ in $\mathbb{R}^m$, cut open the surface along the core curves of its cylinders, rotate the top component of each cylinder $\mathcal{C}_i$ by an amount of $x_ic_i$ and glue back. This construction yields a smooth embedding from the $m$-dimensional torus to $\moduli{}$. See figure 1. This embedding is of particular interest as it allows to see Rel Rel trajectories of horizontally periodic surfaces as linear flows, of which we have a detailed understanding. See Proposition \ref{intertwined}.

\begin{figure}[h]
    \centering
\begin{tikzpicture}
    \begin{scope}[shift={(0,0)}]
    \draw (1,0) rectangle (2,0.5);
    \draw (0,.5) rectangle (2,1);
    \draw (0,1) rectangle (1,1.5);
    \draw (1.5,.25) node{3};
    \draw (1,.75) node{2};
    \draw (.5,1.25) node{1};
    \draw (1,-.5) node{$q = \tau(0,0,0)$};
    \draw[>=latex,->](2.5,.75) .. controls +(45:.5) and +(135:.5) .. (4.5,.75) ;
    \end{scope}
    
    \begin{scope}[shift={(5,0)}]
    \draw (1,0.5) -- (.5,0) -- (1.5,0) -- (2,.5);
    \draw (0,.5) rectangle (2,1);
    \draw (0,1) -- (1,1.5) -- (2,1.5) -- (1,1);
    \draw (1,-.5) node{$\tau(1,0,1/2)$};
    \draw[fill=black,opacity = .1] (1,1) -- (1,1.5) -- (2,1.5) -- cycle;
    \draw[dashed] (1,1) -- (1,1.5);
    \end{scope}
    
    \begin{scope}[shift={(8,0)}]
    \draw (-.5,0.75) node{$=$};
    \draw (1,0.5) -- (.5,0) -- (1.5,0) -- (2,.5);
    \draw (0,.5) rectangle (2,1);
    \draw (0,1) rectangle (1,1.5);
    \draw (1,-.5) node{$\tau (0,0,1/2)$};
    \draw[fill=black,opacity = .1] (0,1) -- (0,1.5) -- (1,1.5) -- cycle;
    \draw[dashed] (0,1) -- (1,1.5);
    \end{scope}

    \end{tikzpicture}
	 \caption{twisting cylinders}
    \label{twistmap}
\end{figure}
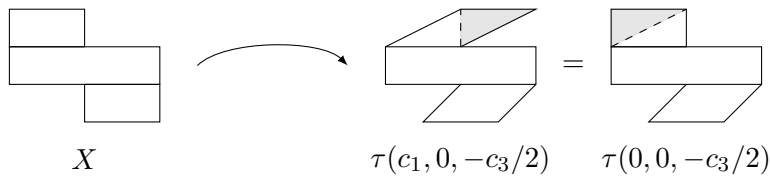

\medskip

More precisely, let $f : S \to X_q$ be a marking on $q$, and define $\xi = \Phi(q,f)$. Recall that the algebraic intersection induces a non-degenerate form $I : H_1(S-\Sigma,\mathbb{Z}) \times H_1(S,\Sigma,\mathbb{Z}) \to \mathbb{Z}$. The core curves of the cylinders can be seen as elements of $H_1(S-\Sigma,\mathbb{Z})$ and we denote by $C_i^{\ast}$ the cocycle $I(f^{-1}\circ \gamma_i,\cdot)$. Notice that the $C_i^{\ast}$ are linearly independent and define $E$ to be the linear subspace of $H^{1}(S,\Sigma,\mathbb{R})$ spanned by the $C_i^{\ast}$ and finally set $\mathcal{C}$ to be the arc-wise component of $\{(q',f') \in \teich \ | \ \Phi(q',f') \in \xi + E\}$ that contains $(q,f)$. 

\begin{prop}
The period map induces a diffeomorphism from $\mathcal{C}$ to $\xi + E$. 
\end{prop}

Now, define a map $\tilde{\tau}$ from $\mathbb{R}^m$ to $\teich$ that sends $(x_1,\cdots,x_m)$ to the surface of $\mathcal{C}$ whose image by the period map is $\xi+ \sum c_ix_iC_i^{\ast}$. In particular, notice that $\tilde{\tau}(0)=(q,f)$. The action of $\mathbb{Z}^m$ on $\mathbb{R}^m$ by translation is intertwined by $\tilde{\tau}$ with the action of the subgroup of $MCG(S,\Sigma)$ generated by Dehn twists about the core curves of the cylinders. There is also a finite subgroup $\Delta$, perhaps trivial, of automorphisms of $q$ that acts on $q$ by permuting cylinders and on the $m$-dimensional torus $\mathbb{T}^m$ by permuting the coordinates. Consequently, there is a map $\tau : \mathbb{T}^m / \Delta$ to $\moduli$ that fits into the commutative diagram depicted in figure \ref{commutativity}. We shall refer to the map $\tau$ as the twist map associated to $q$. Notice that by construction $\Phi \circ \tilde{\tau}$ is affine with constant part is $\xi$ and is injective. For more details on this construction, the reader is referred to Section 6.2 of \cite{hooper2015rel}.  

\begin{figure}[h]
    \centering
    \begin{tikzpicture}
    \draw[>=latex,->] (1,0) -- (2,0) ;
    \draw[>=latex,->] (1,-1.5) -- (2,-1.5);
    \draw[>=latex,->] (0,-.5) -- (0,-1) ;
    \draw[>=latex,->] (3,-.5) -- (3,-1) ;
    \draw (0,0) node{$\mathbb{R}^m$};
    \draw (0,-1.5) node{$\mathbb{T}^m / \Delta $};
    \draw (3,0) node{$\teich$};
    \draw (3,-1.5) node{$\moduli$};
    \draw (1.5,0) node[above]{$\tilde{\tau}$};
    \draw (1.5,-1.5) node[above]{$\tau$};
    \draw (-.1,-.75) node[left]{$pr$};
    \draw (3.1,-.75) node[right]{$\pi$};
    \end{tikzpicture}
    \caption{the twist map}
    \label{commutativity}
\end{figure} 

In the previous construction, the subgroup $\Delta$ depends on the choice of the marking $f : S \to X_q$. Another choice would have produced a group conjugated to $\Delta$. 
\medskip

\subsection{Twist map, real rel flows and affine invariant orbifolds}

Let $\mathcal{M}$ be a nonabsolute affine invariant orbifold that contains the surface $q$. We denote by $k$ the rank of $\M$ and by $r$ the dimension of the foliation $\FM$. This last quantity is known as the Rel of $\M$. We also fix a local model $V$ of $\M$ around $q$, \textit{i.e.} there is an orbifold chart $(\mathcal{V},\varphi,\Gamma)$ around $q$, a point $\mathfrak{q} \in \mathfrak{M}$ such that $\iota(\mathfrak{q})=q$ and a neighborhood $\mathcal{U} \subset \mathfrak{M}$ around $\mathfrak{q}$ such that $\iota(\mathcal{U}) = \varphi(\mathcal{V}\cap V \otimes \mathbb{C})$. We can further assume that for any $v \in \R^m$, we have $\pi \circ \tilde{\tau}(v) \in \M$ if, and only if, $\Phi \circ \tilde{\tau}(v) \in V\otimes \mathbb{C}$. From now on, we assume that the cylinder decomposition of $q$ is \textbf{$\mathcal{F}^\mathcal{M}$-stable}: this means that any horizontal saddle connection of $q$ vanishes as an element of $T^{\ast}_X\mathcal{F}_q^{\mathcal{M}}$. Equivalently, this means that the period of these saddle connections remain unchanged along any deformation tangent to $\FM$. We emphasize that this condition is dependent on the locus $\M$ under consideration. To highlight this fact, consider a Veech surface $q_0$ and let $q$ be a flat cover of $q_0$ \textit{i.e.} there is a holomorphic map $\pi : q \to q_0$ such that $\pi^{\ast} \omega_0 = \omega$. We further assume that $\pi$ has three ramification points, exactly two of which are periodic points (see \cite{Gutkin2003affine} for relevant definitions). Up to applying a rotation, it can be assumed that these two periodic points are horizontally aligned. Finally, let $\mathcal{H}$ be the stratum in which $q$ lives and let  $\M \subset \mathcal{H}$ be its $\GL$-orbit closure. It follows from the assumptions that $\M$ is a rank one affine invariant orbifold, that $q$ is horizontally periodic and that decomposition is $\FM$-stable whenever the third ramification point is not aligned with the two periodic ones but the decomposition is never $\mathcal{F}^{\mathcal{H}}$-stable. Proposition \ref{intertwined} below shows that under the $\FM$-assumption, applying $Rel_v^t$ to $q$ with $v \in \frakK$ simply amounts to twisting its cylinders. On the contrary, the following figure shows an example of a real rel deformation that is not the result of twisting cylinders:

\begin{center}
    \begin{tikzpicture}[scale = 0.7]
    \begin{scope}[shift={(0,0)}]
    \draw (0,0) rectangle (2,1); 
    \draw (.7,1) rectangle (2.7,2);
    \draw (0,0) node{$\times$};
    \draw (.7,0) node{$\bullet$};
    \draw (2,0) node{$\times$};
    \draw (0,1) node{$\times$};
    \draw (.7,1) node{$\bullet$};
    \draw (2,1) node{$\times$};
    \draw (2,2) node{$\times$};
    \draw (.7,2) node{$\bullet$};
    \draw (2.7,1) node{$\bullet$};
    \draw (2.7,2) node{$\bullet$};
     \draw[>=latex,->](3.5,1) .. controls +(45:.5) and +(135:.5) .. (4.5,1) ;
    \end{scope}
    
    \begin{scope}[shift={(5.3,0)}]
    \draw (0,0) rectangle (2,1); 
    \draw (1.6,1) rectangle (3.6,2);
    \draw (0,0) node{$\times$};
    \draw (1.6,0) node{$\bullet$};
    \draw (2,0) node{$\times$};
    \draw (0,1) node{$\times$};
    \draw (1.6,1) node{$\bullet$};
    \draw (2,1) node{$\times$};
    \draw (2,2) node{$\times$};
    \draw (1.6,2) node{$\bullet$};
    \draw (3.6,1) node{$\bullet$};
    \draw (3.6,2) node{$\bullet$};
    
    \end{scope}
 
    \end{tikzpicture}
\end{center}

Let $v \in \frakK$ and let $\sigma_i$ be a cross curve of the $\mathcal{C}_i$ \textit{i.e.} a path that connects a zero of $\omega_q$ on the bottom component of $\mathcal{C}_i$ to another zero on the top component while intersecting the core curve only once. Notice that as $v \in \frakK$ and the decomposition is $\FM$-stable, the value $v(\sigma_i)$ does not depend on the choice of the cross curve. Finally, let $w_v = (v(\sigma_1)c_1^{-1},\cdots,v(\sigma_m)c_m^{-1}) \in \R^m$. We have the following: 

\begin{prop}
The straight-line flow $\phi_{v}^t: x \mapsto x + t\cdot w_v$ on $\R^m$ induces a well defined flow on $\mathbb{T}^m / \Delta$ and for any $t>0$, we have: 

\begin{equation*}\label{intertwined}
    Rel_v^t \circ \tau = \tau \circ \phi_{v}^t.
\end{equation*}
\end{prop}

\begin{proof}
It is know that the only obstruction for Rel trajectories to be defined for arbitrarily long time is the existence of saddle connections becoming shorter and shorter, see Proposition 4.5 in \cite{bainbridge2016horocycle}. This phenomenon is ruled out for $Rel_v^t$ trajectories starting at points in the image of $\tau$ by the $\FM$-stability condition. Now, by definition $\Phi \circ \tilde{\tau}(w_v) = \xi + \sum_{i=1}^m v(\sigma_i)C_i^{\ast}$ and we claim that $v = \sum_{i=1}^m v(\sigma_i)C_i^{\ast}$. To see this, one can use a basis of $H_1(S,\Sigma,\mathbb{C})$ obtained by adjoining to the $\sigma_i$ a family of saddle connections living on the boundaries of the cylinders. Since the cylinder decomposition of $q$ is $\FM$-stable, we know that $v$ vanishes on these saddle connections. Our claim then follows from the fact that the $\sigma_i$ are dual to the $C_i^{\ast}$ with respect to the intersection form $I$. This computation shows that $\tilde{\tau}$ sends the constant vector field $w_v$ to the constant vector field $\Phi^{\ast}v$ and the associated flows are by definition the straight line flow $\phi_{v}^t$ on $\R^m$ and the Rel flow $Rel_v^t$ on $\teich$. It is thus enough to show that $\phi_v$ induces a well defined flow on $\mathbb{T}^m / \Delta$ or equivalently that $\phi_v^t$ commutes with the action of $\Delta$. Recall that $\Delta$ must preserve the labelling of the zeroes. Consequently, if an element of $\Delta$ maps $\mathcal{C}_i$ to $\mathcal{C}_j$, then $\sigma_i$ and $\sigma_j$ can be chosen to both start and end at the same singularity and $\sigma_j  - \sigma_j$ is thus a loop. Since $v \in \mathrm{ker}\rho$, we get that $v(\sigma_i) - v(\sigma_j) = 0$. It is also the case that $c_i=c_j$ and thus ${w_v}_i = {w_v}_j$. This shows that $\Delta$ and $\phi_v^t$ commute.
\end{proof} 

We denote by $\K$ the subset $\{w_v \in \R^m \ | \ v \in \frakK \}$. It is a linear subspace and by construction we have $\Phi \circ \tilde{\tau}(\K) = \xi + \frakK$. We also denote by $\mu = (h_1/c_1,\cdots, h_m/c_m)$ and by $\V = \K \oplus \langle  \mu \rangle$. Notice that $\tau \circ pr(t \cdot \mu)=h_t \cdot q$. Since the horocycle flow commutes with Real Rel flows, using Proposition \ref{intertwined}, we see that $\tau \circ pr(\V)$ is contained in $\M$.  

\begin{prop}\label{VK}
Suppose that the rank of $\M$ is 1. Then, the space $\V$ is a rational. 
\end{prop}

\begin{proof}
Suppose to a contradiction that $\V$ is not rational. Any rational subspace $F$ contained in $\V$ thus has empty interior and since there are only countably many such rational subspaces, Baire category theorem ensures that there is a vector $v \in \V$ that is not contained in any rational subspace of $\V$ and we let $F$ be the smallest rational subspace that contains $v$. By assumption, $F$ is not contained in $\V$. By the Kronecker theorem (see Proposition 7 on page 74 of \cite{bourbaki1966elements}) we know that $pr \{t \cdot v \ | t > 0 \}$ is dense in $pr(F)$. Since $\tau \circ pr \{t \cdot v \ | t > 0 \}$ is contained in $\M$ and that $\M$ is closed, then $\tau \circ pr(F) = \pi \circ \tilde{\tau}(F)$ is also contained in $\M$. This means that $\Phi \circ \tilde{\tau}(F)$ is contained in $V \otimes \mathbb{C}$ (cf. the assumption made at the beginning of Section 3.2). Then, since the core curves of the cylinders do not intersect each other, the image by $\rho$ of the linear subspace $(\Phi \circ \tilde{\tau}(\R^m) - \xi) \cap V \otimes \mathbb{C}$ is an isotropic subspace of $\rho(V \otimes \mathbb{C})$ and we claim its dimension is at least $2$. Indeed, let $f \in F - \V$ and suppose to a contradiction that there is a $t \in \mathbb{R}$ such that $\rho ( \Phi \circ \tilde{\tau}(f)- \xi) = t\cdot \rho ( \Phi \circ \tilde{\tau}(\mu)- \xi)$. Using the fact that the map $\Phi \circ \tilde{\tau}(\cdot) - \xi$ is linear, it follows that there is a $k \in \frakK$ such that $\Phi\circ \tilde{\tau}(f - t\cdot \mu) = \xi + k$ and the latter is equal to $\Phi \circ \tilde{\tau}(w_k)$. Since the map $\Phi \circ \tilde{\tau}$ is injective, this shows that $f = w_k + t\cdot \mu \in \V$, a contradiction. It follows that the images by $\rho$ of $\Phi \circ \tilde{\tau} (f) - \xi$ and $\Phi \circ \tilde{\tau} (\mu)-\xi$ are linearly independent vectors in $\rho(V \otimes \mathbb{C})$. This is in contradiction with the fact that $\rho(V \otimes \mathbb{C})$ should be a symplectic space of dimension 2 ($\M$ has rank 1).  
\end{proof} 

We now define the following property: 

\begin{dfn}[Property $\mathcal{P}$]
The surface $q$ is said to have property $\mathcal{P}$ if it has a horizontal $\FM$-stable cylinder decomposition and $\K$ is not a rational subspace of $\R^m$. By extension, we say that an affine invariant orbifold $\M$ has property $\mathcal{P}$ if it contains a surface that has property $\mathcal{P}$.
\end{dfn}

\section{Property $\mathcal{P}$ and density of the leaves}\label{p}

This section is dedicated to the proof of Theorem \ref{A}.

\begin{prop}\label{containshorocycle}
Let $\mathcal{M}$ be a nonabsolute rank one affine invariant orbifold and suppose that it contains a surface  $q$ that has property $\mathcal{P}$. Then the closure of $\FM$ contains the horocycle orbit of $q$. 
\end{prop}

\begin{proof}
It is an immediate consequence of Proposition \ref{VK} that if $\K$ is not rational then the smallest rational subspace that contains $\K$ is $\V$. By the Kronecker theorem, this implies that $pr(\K)$ is dense in $pr(\V)$. We have by construction $\tau \circ pr(\K) \subset \FM_q$ and thus $\tau \circ pr(\V)$ is contained in the closure of $\FM_q$. But since $\mu \in \V$ and $\tau \circ pr(t\cdot \mu) = h_t \cdot q$, we can conclude. 
\end{proof}

\begin{prop}\label{whole}
Let $\mathcal{M}$ be a nonabsolute rank 1 affine invariant orbifold and $q \in \mathcal{M}$. If the horocycle orbit of $q$ is contained in the closure of its isoperiodic leaf, then so is its $\SL$-orbit.
\end{prop}

\begin{proof}
We will need the following lemma:

\begin{lem}
The $\M$-isoperiodic leaf of $q$ contains a surface fixed by a non trivial hyperbolic matrix.
\end{lem}

\begin{proof}[proof of the lemma]
It is proved in, for instance, \cite{wright2014field} that any affine invariant orbifold contains a surface fixed by a non trivial hyperbolic matrix, using a closing lemma for the geodesic flow. Let $q_1$ be such a surface associated to a hyperbolic matrix $g$. Now equation \eqref{product2} applied to $q_1$ shows that there is a matrix $b$ that sends $q_1$ to the $\M$-isoperiodic leaf of $q$. The surface $b \cdot q_1$ is fixed by the hyperbolic matrix $bgb^{-1}$. 
\end{proof}

Now, pick a surface $q_0 \in \FM_q$ as in the previous lemma with $g \cdot q_0 = q_0$ for some hyperbolic matrix $g$. Denote by $E=\{h \in SL_2(\R) \ | \ h \cdot q_0 \in \overline{\FM_{q_0}} \}$. Using equation \eqref{commutation} and the fact that the closure of $\FM_q = \FM_{q_0}$ is saturated by $\FM$ we show that $U \subset E$. Since $g \cdot q_0 = q_0$ it follows also that $gUg^{-1} \subset E$. If $U$ and $gUg^{-1}$ are equal, it means that $g$ preserves the horizontal direction of $q_0$. Up to replacing $q_0$ and $q$ by $h_t \cdot q_0$ and $h_t \cdot q$ for a well chosen $t$, we can assume that $g \in A$. In that case, we can use Proposition 5.2 in \cite{hooper2015rel} to conclude that that $E=SL_2(\R)$. Otherwise $gUg^{-1}$ is distinct from $U$ and since $SL_2(\R)$ is generated as a monoid by any two of its unipotent subgroups, we have the smallest set that contains $U$ and $gUg^{-1}$ and that is invariant by multiplication is $\SL$ itself. To conclude that $E=\SL$, notice that $E$ is invariant by multiplication, which follows again from equation \eqref{commutation} and the fact that the closure of $\FM_{q_0}$ is saturated by the $\M$-isoperiodic foliation. Finally, for any $h \in \SL$ we have that $h\cdot q \in \FM_q$ if and only if $h\cdot q_0 \in \FM_{q_0}$ which concludes the proof. 
\end{proof}

We can now prove: 
\begin{thm}[Theorem \ref{A}]\label{crit1}
Let $\mathcal{M}$ be a nonabsolute rank one affine invariant orbifold having property $\mathcal{P}$. Then the $\M$-isoperiodic foliation is projectively minimal.
\end{thm}

\begin{proof}
Let $q$ be a surface in $\M$ that has property $\mathcal{P}$. According to Propositions \ref{containshorocycle} and \ref{whole}, $\SL\cdot q$ is contained in the closure of $\FM_q$. Using equation \eqref{commutation} and the fact the closure of $\FM_q$ saturated by $\FM$, we deduce that the closure $\FM_q$ is invariant under the action of $\SL$. Using equation \eqref{product2}, we conclude that the $\M$-isoperiodic leaf of $q$ is projectively dense. Finally, another utilisation of equation \eqref{product2} together with \eqref{commutation} proves that there is an element of $\GL$ that sends the $\M$-isoperiodic leaf of $q$ to the $\M$-isoperiodic leaf of any other surface in $\M$. Since this action is continuous, all the leaves are projectively dense. 
\end{proof}

In Section 5, we will give examples of rank 1 affine invariant orbifolds that have property $\mathcal{P}$. However, not all of them have it, as shows the following example. 

\begin{prop}\label{counterexample} Let $q_0$ be a Veech surface, that is a surface whose $\GL$-orbit is closed, and let $p: q \to q_0$ be a flat cover: a holomorphic map such that $p^{\ast}\omega_{q_0} = \omega_q$. Finally, let $\M$ be the $\GL$-orbit closure of $q$. Then, $\M$ is a rank 1 affine invariant orbifold such that $\FM_q$ is closed. In particular, if $\M$ is an arithmetic rank 1 affine invariant orbifold \textit{i.e.} $k(\M)=\mathbb{Q}$, then the leaves of $\FM$ are closed. 
\end{prop}

\begin{proof} We start by proving the first assertion. Let $\M_0$ be the $\GL$-orbit of $q_0$. It is easy to show that any surface in $\M$ is also a flat cover over a surface in $\M_0$. Thus there is a $\GL$-equivariant continuous map $\Psi$ from $\M$ to $\M_0$ that maps a surface to the unique surface it covers, and this map is continuous. To conclude, it remains to notice that the leaf $\FM_q$ is a connected component of the preimage of $q_0$ by $\Psi$, which is closed. 

\medskip

For the second assertion, notice that in this case, $\M$ contains a surface whose periods are rational and thus is a flat cover over a torus, which is a Veech surface. Up to flowing along $\FM$, we get a new surface that still is a flat cover of a torus and whose $\GL$-orbit closure is $\M$. The second assertion thus follows from the first one.   
\end{proof}

\textbf{Remark.} The affine invariant orbifold $\M$ in Proposition \ref{counterexample} is nonabsolute if, and only if, the set of branch points of $p$ contains a non periodic point. See \cite{Gutkin2003affine} for more details. 

\section{Examples of rank 1 affine invariant orbifolds with property $\mathcal{P}$}\label{examples}

\subsection{In genus 2}

\begin{prop}\label{pforeigen2}
Let $\M$ be a nonabsolute nonarithmetic rank one affine invariant orbifold contained in $\mathcal{H}(1,1)$. Then $\M$ has property $\mathcal{P}$.   
\end{prop}

\begin{proof}
Let $q$ be a horizontally periodic surface in $\M$. Such a surface is given, for example, by Corollary 6 of \cite{smillie2004minimal}. Notice that here $\frakK = \mathrm{ker} \rho$ since the leaves of $\mathcal{F}$ have dimension 1. In particular, $q$ is $\FM$ stable in that case if, and only if, it does not have a horizontal saddle connection connecting the two singularities. Notice that  applying $Rel_{i\cdot v}^t$ for a small $v \in \frakK$ to $q$ would break any potential such saddle connection while not changing the fact that the surface is horizontally periodic. This means that up to replacing $q$ with $Rel_{i\cdot v}^t q$ we can assume that the horizontal decomposition is $\FM$-stable. The surface $q$ thus has a cylinder decomposition as follows: 

\begin{center}
    \begin{tikzpicture}
    
    \draw (1,0) rectangle (2,.5); 
    \draw (0,.5) rectangle (2,1);
    \draw (0,1) rectangle (1,1.5); 
    
    \draw (1,0) node{$\bullet$};
    \draw (2,0) node{$\bullet$};
    \draw (0,1) node{$\bullet$};
    \draw (1,1) node{$\bullet$};
    \draw (2,1) node{$\bullet$};
    
    \draw (0,.5) node{$\times$};
    \draw (1,.5) node{$\times$};
    \draw (2,.5) node{$\times$};
    \draw (0,1.5) node{$\times$};
    \draw (1,1.5) node{$\times$};
    
    \draw (1.5,.25) node{1};
    \draw (1,.75) node{2};
    \draw (.5,1.25) node{3};
    
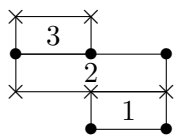
    \draw (1,-1) node{fig. The only stable cylinder decomposition in genus 2};
    
    \end{tikzpicture}
\end{center}

This can be seen by analyzing the saddle connection diagram of $q$, which is very simple in $\mathcal{H}(1,1)$. Now, let $v = I(\gamma,\cdot) \in H^1(X_q,\Sigma_q,\mathbb{C})$ where $\gamma$ is a homotopically trivial loop in $S$ that circles clockwise the singularity $\bullet$ (it is thus non trivial as an element of $H_1(X_q - \Sigma_q)$). Since $\gamma$ is homotopically trivial, we have that $I(\gamma,\cdot) \in \mathrm{ker}\rho \simeq \mathrm{ker}\rho_q$ and $w_v = (c_1^{-1},-c_2^{-1},c_3^{-1})$ where $c_i$ is the length of the core curve of the $\mathcal{C}_i$, as in Section 3. Now, if $\K$ is rational then all the entries of $w_v$ are pair-wise commensurable and by Proposition \ref{field}, this implies that $\M$ is arithmetic, a contradiction. So $q$ has property $\mathcal{P}$. 
\end{proof}

We recall that the nonabsolute rank one affine invariant orbifolds in genus 2 have been classified by McMullen in \cite{mcmullen2007dynamics}, and they are connected components of Prym eigenform loci. In this case $\mathcal{F}$ and $\FM$ are the same as the $\Omega E_D(1,1)$ are saturated by the isoperiodic foliation. These remarks, together with Theorem \ref{A} and Proposition \ref{counterexample} then implies: 

\begin{cor}[Theorem \ref{B}, genus 2 case]\label{Bgenus2}
Let $\M$ be a connected component of $\Omega E_D(1,1)$. Then either all the leaves $\FM$ are closed or all the leaves are projectively dense in $\M$. The latter case occurs if, and only if, $D$ is not a square. 
\end{cor}

McMullen gave in \cite{mcmullen2007dynamics} a description of the set of prym eigenforms of genus 2 as a $\mathbb{C}^{\ast}$-bundle over a Zariski open set of the Hilbert modular surface $X_D$. In that perspective, it is interesting to compare Corollary \ref{Bgenus2} with Theorem 9.3 in \cite{mcmullen2007foliations}. 

\subsection{Prym eigenform loci in genus 3}

\begin{prop}\label{pforeigen}
Let $\kappa = (2,2)^{odd}, (2,1,1)$ or $(1,1,1,1)$ and let $\M$ be a connected component of $\Omega E_D(\kappa)$. Then $\M$ has property $\mathcal{P}$ if and only if D is not a square. 
\end{prop}

\begin{proof}
Since $k(\M) = \mathbb{Q}(\sqrt{D})$, the case where $D$ is a square is thus contained in Proposition \ref{counterexample}. We assume from now on that $D$ is not a square. We treat separately the different strata:

\begin{itemize}
\item If $\M$ is contained in $\Omega E_D^{odd}(2,2)$, Theorem $B$ of \cite{lanneau2014connected} states that $\M$ contains a translation surface whose cylinder decomposition is as follows:
    
    \begin{center}
\begin{tikzpicture}

\draw (1,0) rectangle (2,.5) ;
\draw (0,.5) rectangle (3,1) ; 
\draw (0,1) rectangle (1,1.5); 
\draw (2,1) rectangle (3,1.5); 

\draw (0,.5) node{$\times$};
\draw (1,.5) node{$\times$};
\draw (2,.5) node{$\times$};
\draw (3,.5) node{$\times$};
\draw (0,1.5) node{$\times$};
\draw (1,1.5) node{$\times$};
\draw (2,1.5) node{$\times$};
\draw (3,1.5) node{$\times$};

\draw (1,0) node{$\bullet$};
\draw (2,0) node{$\bullet$};
\draw (0,1) node{$\bullet$};
\draw (1,1) node{$\bullet$};
\draw (2,1) node{$\bullet$};
\draw (3,1) node{$\bullet$};

\draw (1.5,0.25) node{1};
\draw (1.5,0.75) node{2};
\draw (.5,1.25) node{3};
\draw (2.5,1.25) node{4};

\draw (.5,.25) node[scale = 0.7] {A}; 
\draw (.5,1.75) node[scale = 0.7] {A}; 
\draw (2.5,.25) node[scale = 0.7] {B}; 
\draw (2.5,1.75) node[scale = 0.7] {B}; 

\end{tikzpicture}
\end{center}

The Prym involution $\lambda$ exchanges cylinder $1$ and $3$ and fixes cylinder 2 and 4 (so $c_1 = c_3$). Notice that the cylinder decomposition is $\FM$-stable (and even $\mathcal{F}$-stable) and that $\frakK$ has dimension $1$. We claim that $u = (c^{-1}_1,-c^{-1}_2,c_3^{-1},c^{-1}_4) \in \K$, we use the same argument as in the proof of Proposition \ref{pforeigen2}. If $\K$ is rational then all the entries of $u$ are pair-wise commensurable and by Proposition \ref{field}, $D$ is a square, a contradiction. So $q$ has property $\mathcal{P}$. 

\item If $\M$ is contained in $\Omega E_D(2,1,1)$, let $q \in \M$ and denote by $\lambda$ the Prym involution on $q$. Notice that $\lambda$ cannot fix a singularity of order 1 (see for instance Lemma 1 in \cite{lanneau2004hyperelliptic}) thus $\lambda$ swaps the order 1 singularities and it follows from Proposition \ref{localmodelleafforeigen} that $\frakK$ has dimension  1. By Theorem $B$ of \cite{lanneau2014connected}, we can assume that the cylinder diagram of $q$ is as follows: 

\begin{center}
\begin{tikzpicture}

\draw (0,0) rectangle (1,.5); 
\draw (0,.5) rectangle (2,1); 
\draw (1,1) rectangle (2,1.5); 
\draw (1,1.5) rectangle (3,2);
\draw (2,2) rectangle (3,2.5); 

\draw (0,0) node{$\bullet$}; 
\draw (1,0) node{$\bullet$};
\draw (0,1) node{$\bullet$}; 
\draw (1,1) node{$\bullet$};
\draw (2,1) node{$\bullet$}; 

\draw (1,1.5) node{$\times$}; 
\draw (2,1.5) node{$\times$};
\draw (3,1.5) node{$\times$};
\draw (2,2.5) node{$\times$};
\draw (3,2.5) node{$\times$};

\draw (0,.5) node{$\otimes$};
\draw (1,.5) node{$\otimes$};
\draw (2,.5) node{$\otimes$};
\draw (1,2) node{$\otimes$};
\draw (2,2) node{$\otimes$};
\draw (3,2) node{$\otimes$};

\draw (.5,-.25) node[scale = .7]{A};
\draw (.5,1.25) node[scale = .7]{A};
\draw (2.5,1.25) node[scale = .7]{B};
\draw (2.5,2.75) node[scale = .7]{B};

\draw (.5,.25) node{1};
\draw (1,.75) node{2};
\draw (1.5,1.25) node{3};
\draw (2,1.75) node{4};
\draw (2.5,2.25) node{5};

\end{tikzpicture}
\end{center}

Here the prym involution $\lambda$ swaps the cylinders 1 with 5, 2 with 4, fixes the cylinder 3 and fixes the singularity $\otimes$. Let $v = I(\gamma-\lambda_{\ast}\gamma,\cdot) \in H^1(X_q,\Sigma_q,\mathbb{C})$ where $\gamma$ is an homotopically trivial loop in $S$ that circles clockwise the singularity $\bullet$ (it is thus non trivial as an element of $H_1(X_q - \Sigma_q)$). Since $\gamma$ is homotopically trivial, we have that $v \in \mathrm{ker}\rho \simeq \mathrm{ker}\rho_q$ and thus $v \in \frakK$. It follows from Proposition \ref{localmodelleafforeigen} that $w_v = (-c_1^{-1},c_2^{-1},-2c_3^{-1},c_4^{-1},-c_5^{-1}) \in \K$. If $\K$ is rational then all the entries of $u$ are pair-wise commensurable and by Proposition \ref{field}, $D$ is a square, a contradiction. So $q$ has property $\mathcal{P}$.

\item Assume $\M$ is contained in $\Omega E_D(1,1,1,1)$, let $q \in \M$ and denote by $\lambda$ the Prym involution on $q$. Notice that $\lambda$ cannot fix a singularity of order 1 (see for instance Lemma 1 in \cite{lanneau2004hyperelliptic}) thus there are two pairs of singularities exchanged by the prym involution. It follows from Proposition \ref{localmodelleafforeigen} that $\frakK$ has dimension $2$. By Corollary 6 of \cite{smillie2004minimal}, we can assume that $q$ is horizontally periodic and up to replacing $q$ by $Rel_{i\cdot v}^t(q)$ for a small $v \in \frakK$, we can assume that no horizontal saddle connection on $q$ joins two different singularities of $\omega_q$. The horizontal cylinder decomposition of $q$ is thus $\FM$-stable. This has two consequences. First, we deduce that $q$ has exactly $8$ horizontal saddle connections, 2 for each of the singularities. Then, since each singularity appears on the bottom component of some cylinder of $q$, we conclude that $q$ has at least 4 horizontal cylinders. The boundary components of these cylinders are made either of 1 or two 2 saddle connections and there are exactly 8 boundary components made of exactly one saddle connection. Since the Prym involution maps a top component of a cylinder to the bottom one of some other cylinder, we deduce that there are at least 4 cylinders whose bottom component is made of a simple saddle connection. Finally, as each saddle connection appears on the bottom component of some cylinder of $q$ and since there are 8 of them, we see that $q$ has at least 6 cylinders. On the other hand, we also know that there are at most $6$ cylinders. See for instance Lemma C.1 in \cite{aulicino2016rank1} and thus $q$ has exactly 6 cylinders. Furthermore, we can always assume $\lambda$ fixes at least one of them. Indeed, if there is a horizontal segment $\sigma$ that joins two fixed points, then the concatenation $\gamma$ of $\sigma$ and $\lambda(\sigma)$ is a horizontal loop that has to be homologous to the core curve of a horizontal cylinder $\mathcal{C}$ (this loop does not cross any singularity as otherwise there would be a horizontal saddle connection connecting two different singularities) By construction, $\gamma$ is invariant by $\lambda$ and thus the cylinder $\mathcal{C}$ is itself invariant by $\lambda$. Conversely, if a cylinder $\mathcal{C}$ is preserved by $\lambda$, the core curve $\gamma$ of $\mathcal{C}$ contains two fixed points of $\lambda$. This is due to the fact that an orientation-reversing homeomorphism of the circle has at least two fixed points. Suppose now that $\lambda$ does not fix any of horizontal cylinders. There are thus 3 pairs of cylinders exchanged by $\lambda$ and an application of Riemann-Hurwtiz reveals that $\lambda$ has exactly $4$ fixed points. By the pigeon hole principle, there is then at least one pair of cylinders $\mathcal{C}$ and $\mathcal{C}'$ that contain two fixed points $p_1$ and $p_2$ of $\lambda$ on the union of their boundary components. Because these cylinders are exchanged by $\lambda$, it means that both $\mathcal{C}$ and $\mathcal{C}'$ have $p_1$ and $p_2$ on their boundaries. By what was said before, it is not the case that $p_1$ and $p_2$ are joined by a horizontal segment and then $\mathcal{C}$ has $p_1$ on one of its boundary components and $p_2$ on the other (likewise for $\mathcal{C}'$). Choose a segment $\sigma$ that connects $p_1$ and $p_2$ through $\mathcal{C}$. This segment does not encounter any singularity as the interior of a cylinder does not contain singularities. We can thus apply the same argument as before but this time in the direction of $\sigma$ to find at least one cylinder fixed by $\lambda$. Up to rotating the surface $q$, we can assume that this cylinder is horizontal. By the complete periodicity of rank 1 affine invariant orbifolds (Theorem 1.5 of \cite{wright2015cylinder}), we know that the horizontal direction is again periodic and up to applying Rel, we can assume again that the the horizontal direction is also $\FM$-stable and thus our analysis made at the beginning of the proof shows there are 6 horizontal cylinders. For parity reasons, $\lambda$ has to fix another horizontal cylinder. In conclusion, we can assume that $q$ has a $\FM$-stable cylinder decomposition made of 6 horizontal cylinders, exactly two of which are fixed by the Prym involution. The list of cylinder diagrams made of 6 cylinders in genus 3 has been established in \cite{aulicino2016rank1}. We reproduce it here:

\begin{center}
    \begin{tikzpicture}
    
    \begin{scope}[scale=0.5]
    
    \draw (0,2) -- (0,4) -- (-.5,5) -- (.5,5) -- (1,4) -- (1.5,5) -- (2.5,5) -- (2,4) -- (2,3) -- (3,3) -- (3,0) -- (2,0) -- (2,1) -- (1,1) -- (1,2) -- cycle;    
    
    \draw (0,2) node[fill=black!40,scale=.5,rectangle]{};
    \draw (0,3) node[fill=black!40,scale=.4,circle]{};
    \draw (0,4) node[fill=black,scale=.4,circle]{};
    \draw (-.5,5) node[fill=black!40,scale=.5,rectangle]{};
    \draw (.5,5) node[fill=black!40,scale=.5,rectangle]{};
    \draw (1,4) node[fill=black,scale=.4,circle]{};
    \draw (1.5,5) node[fill=black,scale=.5,rectangle]{};
    \draw (2.5,5) node[fill=black,scale=.5,rectangle]{};
    \draw (2,4)  node[fill=black,scale=.4,circle]{};
    \draw (2,3) node[fill=black!40,scale=.4,circle]{};
    \draw (3,3) node[fill=black!40,scale=.4,circle]{};
    \draw (3,2) node[fill=black!40,scale=.5,rectangle]{};
    \draw (3,1) node[fill=black,scale=.5,rectangle]{};
    \draw (3,0) node[fill=black!40,scale=.4,circle]{};
    \draw (2,0) node[fill=black!40,scale=.4,circle]{};
    \draw (2,1) node[fill=black,scale=.5,rectangle]{};
    \draw (1,1) node[fill=black,scale=.5,rectangle]{};
    \draw (1,2) node[fill=black!40,scale=.5,rectangle]{};
    
    \draw (.25,4.5) node[scale=.5]{1};
    \draw (1.75,4.5) node[scale=.5]{2};
    \draw (1,3.5) node[scale=.5]{3};
    \draw (1.5,2.5) node[scale=.5]{4};
    \draw (2,1.5) node[scale=.5]{5};
    \draw (2.5,.5) node[scale=.5]{6};
    
    \draw (.5,1.75) node[scale=0.6]{A};
    \draw (0,5.25) node[scale=0.6]{A};
    \draw (1.5,.75) node[scale=0.6]{B};
    \draw (2,5.25) node[scale=0.6]{B};
    \draw (2.5,-.25) node[scale=0.6]{C};
    \draw (2.5,3.25) node[scale=0.6]{C};
    
    \draw (1.5,-1.5) node{1.};
    
    \end{scope}
        
    \begin{scope}[shift = {(3,0)},scale=0.5]
    
    \draw (0,2) -- (0,3) -- (0,4) -- (1,4) -- (1,5) -- (2,5) -- (2,4) -- (2,3) -- (2.5,4) -- (3.5,4) -- (3,3) -- (3,2) -- (3,1) -- (2,1) -- (2,0) -- (1,0) -- (1,1) -- (1,2) -- cycle;   
        
    \draw (0,2) node[fill=black!40,scale=.5,rectangle]{};
    \draw(0,3) node[fill=black!40,scale=.4,circle]{};
    \draw(0,4) node[fill=black,scale=.4,circle]{};
    \draw(1,4) node[fill=black,scale=.4,circle]{};
    \draw(1,5) node[fill=black,scale=.5,rectangle]{};
    \draw(2,5) node[fill=black,scale=.5,rectangle]{};
    \draw(2,4) node[fill=black,scale=.4,circle]{};
    \draw(2,3) node[fill=black!40,scale=.4,circle]{};
    \draw(2.5,4) node[fill=black!40,scale=.5,rectangle]{};
    \draw(3.5,4) node[fill=black!40,scale=.5,rectangle]{};
    \draw(3,3) node[fill=black!40,scale=.4,circle]{};
    \draw(3,2) node[fill=black!40,scale=.5,rectangle]{};
    \draw(3,1) node[fill=black,scale=.5,rectangle]{};
    \draw(2,1) node[fill=black,scale=.5,rectangle]{};
    \draw(2,0) node[fill=black,scale=.4,circle]{};
    \draw(1,0) node[fill=black,scale=.4,circle]{};
    \draw(1,1) node[fill=black,scale=.5,rectangle]{};
    \draw(1,2) node[fill=black!40,scale=.5,rectangle]{};
    
    \draw (1.5,4.5) node[scale=.5]{1};
    \draw (2.75,3.5) node[scale=.5]{2};
    \draw (1,3.5) node[scale=.5]{3};
    \draw (1.5,2.5) node[scale=.5]{4};
    \draw (2,1.5) node[scale=.5]{5};
    \draw (1.5,.5) node[scale=.5]{6};
    
    \draw (1.5,-.25) node[scale=0.6]{A};
    \draw (.5,4.25) node[scale=0.6]{A};
    \draw (2.5,.75) node[scale=0.6]{B};
    \draw (1.5,5.25) node[scale=0.6]{B};
    \draw (.5,1.75) node[scale=0.6]{C};
    \draw (3,4.25) node[scale=0.6]{C};
    
    \draw (1.5,-1.5) node{2.};
    
    \end{scope}
    
    \begin{scope}[shift = {(6,0)},scale=0.5]
    
    \draw (0,0) -- (0,1) -- (-.5,2) -- (.5,2) -- (1,1) -- (1.5,2) -- (1.5,3) -- (1,4) -- (2,4) -- (2.5,3) -- (3,4) -- (4,4) -- (3.5,3) -- (3.5,2) -- (2.5,2) -- (2,1) -- (2,0) -- (1,0) -- cycle;
    
    \draw (0,0) node[fill=black,scale=.4,circle]{};
    \draw (0,1) node[fill=black!40,scale=.4,circle]{};
    \draw (-.5,2) node[fill=black,scale=.4,circle]{};
    \draw (.5,2) node[fill=black,scale=.4,circle]{};
    \draw (1,1) node[fill=black!40,scale=.4,circle]{};
    \draw (1.5,2) node[fill=black!40,scale=.5,rectangle]{};
    \draw (1.5,3) node[fill=black,scale=.5,rectangle]{};
    \draw (1,4) node[fill=black,scale=.4,circle]{};
    \draw (2,4) node[fill=black,scale=.4,circle]{};
    \draw (2.5,3) node[fill=black,scale=.5,rectangle]{};
    \draw(3,4) node[fill=black!40,scale=.5,rectangle]{};
    \draw (4,4) node[fill=black!40,scale=.5,rectangle]{};
    \draw (3.5,3) node[fill=black,scale=.5,rectangle]{};
    \draw (3.5,2) node[fill=black!40,scale=.5,rectangle]{};
    \draw (2.5,2) node[fill=black!40,scale=.5,rectangle]{};
    \draw(2,1) node[fill=black!40,scale=.4,circle]{};
    \draw (2,0) node[fill=black,scale=.4,circle]{};
    \draw (1,0) node[fill=black,scale=.4,circle]{};
    
    \draw (1.75,3.5) node[scale=.5]{1};
    \draw (3.25,3.5) node[scale=.5]{2};
    \draw (2.5,2.5) node[scale=.5]{3};
    \draw (.25,1.5) node[scale=.5]{4};
    \draw (1.75,1.5) node[scale=.5]{5};
    \draw (1,.5) node[scale=.5]{6};
   
    \draw (.5,-.25) node[scale=0.6]{A};
    \draw (0,2.25) node[scale=0.6]{A};
    \draw (1.5,-.25) node[scale=0.6]{B};
    \draw (1.5,4.25) node[scale=0.6]{B};
    \draw (3,1.75) node[scale=0.6]{C};
    \draw (3.5,4.25) node[scale=0.6]{C};
    
    \draw (1.5,-1.5) node{3.};        
    
    \end{scope}
        
    \begin{scope}[shift = {(9,0)},scale=0.5]
    
    \draw (0,0) -- (0,1) -- (-.5,2) -- (.5,2) -- (1,1) -- (1.5,2) -- (1.5,3) -- (1,4) -- (2,4) -- (2.5,3) -- (3,4) -- (4,4) -- (3.5,3) -- (3.5,2) -- (2.5,2) -- (2,1) -- (2,0) -- (1,0) -- cycle;
    
    \draw (0,0) node[fill=black,scale=.4,circle]{};
    \draw (0,1) node[fill=black!40,scale=.4,circle]{};
    \draw (-.5,2) node[fill=black!40,scale=.5,rectangle]{};
    \draw (.5,2) node[fill=black!40,scale=.5,rectangle]{};
    \draw (1,1) node[fill=black!40,scale=.4,circle]{};
    \draw (1.5,2) node[fill=black!40,scale=.5,rectangle]{};
    \draw (1.5,3) node[fill=black,scale=.5,rectangle]{};
    \draw (1,4) node[fill=black,scale=.4,circle]{};
    \draw (2,4) node[fill=black,scale=.4,circle]{};
    \draw (2.5,3) node[fill=black,scale=.5,rectangle]{};
    \draw(3,4) node[fill=black,scale=.4,circle]{};
    \draw (4,4) node[fill=black,scale=.4,circle]{};
    \draw (3.5,3) node[fill=black,scale=.5,rectangle]{};
    \draw (3.5,2) node[fill=black!40,scale=.5,rectangle]{};
    \draw (2.5,2) node[fill=black!40,scale=.5,rectangle]{};
    \draw(2,1) node[fill=black!40,scale=.4,circle]{};
    \draw (2,0) node[fill=black,scale=.4,circle]{};
    \draw (1,0) node[fill=black,scale=.4,circle]{};
   
    \draw (1.75,3.5) node[scale=.5]{1};
    \draw (3.25,3.5) node[scale=.5]{2};
    \draw (2.5,2.5) node[scale=.5]{3};
    \draw (.25,1.5) node[scale=.5]{4};
    \draw (1.75,1.5) node[scale=.5]{5};
    \draw (1,.5) node[scale=.5]{6};
   
    \draw (.5,-.25) node[scale=0.6]{B};
    \draw (0,2.25) node[scale=0.6]{A};
    \draw (1.5,-.25) node[scale=0.6]{C};
    \draw (1.5,4.25) node[scale=0.6]{B};
    \draw (3,1.75) node[scale=0.6]{A};
    \draw (3.5,4.25) node[scale=0.6]{C};
  
    \draw (1.5,-1.5) node{4.};         
    
    \end{scope}
    \end{tikzpicture}
\end{center}

In order to prove that $q$ has property $\mathcal{P}$, we claim that it is enough to find a pair $u=(u_1,\cdots,u_6)$ and $v=(v_1,\cdots,v_6)$ in $\K$ together with an integer $j \in \{1,\cdots,6\}$ such that the entries of $u$ are not all pairwise commensurable, that $u_j = 0$ and that $v_j \neq 0$ (we recall that we use the notation of section 3) Indeed, if such that triple $(u,v,j)$ exists, suppose to a contradiction that $\K$ is a rational subspace of $\R^6$ and let $V_u$ be the smallest rational subspace that contains $u$. We thus have $V_u \subset \K$. The dimension of $V_u$ is at least 2 as otherwise $u$ would be colinear to a rational vector, which would imply that its entries are all pair-wise commensurable. But since the dimension of $\K$ is 2, this means that $V_u = \K$.  In particular, any linear equation with rational coefficients satisfied by the entries of $u$ is satisfied by all the vectors of $\K$. But $u_j=0$ is such an equation that is not satisfied by $v$. We provide such a triple $(u,v,j)$ in all the cases. The vectors $u$ and $v$ will be of the form $w_{\nu}$ with $\nu = I(\gamma - \lambda_{\ast}(\gamma), \cdot) \in \mathrm{ker}\rho \simeq \mathrm{ker} \rho_q$ where $\gamma$ is a homotopically trivial loop that circles one of the singularities. By Proposition \ref{localmodelleafforeigen}, this ensures that they will belong to $\K$. 

\begin{enumerate}
    \item In that case, the Prym involution fixes the cylinders 2 and 4, exchanges 3 with 5 and 1 with 6. We can choose $u = (-c^{-1}_1,0,-c^{-1}_3,2c^{-1}_4,-c^{-1}_5,-c^{-1}_6)$, $v = (c^{-1}_1,2c^{-1}_2,-c^{-1}_3,0,-c^{-1}_5,c^{-1}_6)$ and $j=2$.
    \item In that case, the Prym involution fixes the cylinders 2 and 4, exchanges 3 with 5 and 1 with 6. We can choose $u=(0,-2c^{-1}_2,-c^{-1}_3,2c^{-1}_4,-c^{-1}_5,0)$, $v=(2c^{-1}_1,0,-c^{-1}_3,0,-c^{-1}_5,2c^{-1}_6)$ and $j=1$.
    \item In that case, the Prym involution preserves the cylinders 1 and 5, it swaps 3 with 6 and 4 with 2. We can choose $v=(0,c^{-1}_2,-c^{-1}_3,c^{-1}_4,2 c^{-1}_5, -c^{-1}_6)$, $v=(2c^{-1}_1,c^{-1}_2,-c^{-1}_3,c^{-1}_4,0,-c^{-1}_6)$ and $j=1$
    \item In that last case, there are two possibilities for the Prym involution. The first possibility is that it preserves cylinder 4 and 5 and swaps 3 with 6 and 1 with 2. The second possibility is that the Prym involution fixes the cylinders 1 and 2, swaps 3 with 6 and 4 with 5. In both cases, we can choose $(0,0,-c^{-1}_3,2c^{-1}_4,2c^{-1}_5,-c^{-1}_6) $, $v=(2c^{-1}_1,2c^{-1}_2,-c^{-1}_3,0,0,-c^{-1}_6)$ and $j=1$
\end{enumerate}

\end{itemize}
\end{proof}

Proposition \ref{pforeigen} together with Theorem \ref{A} and Proposition \ref{counterexample} implies the following: 

\begin{cor}[Theorem \ref{B}, genus 3 case]
Let $\kappa = (2,2)^{odd}, (2,1,1)$ or $(1,1,1,1)$ and let $\M$ be a connected component of $\Omega E_D(\kappa)$. Then either all the leaves of $\FM$ are all closed or all the leaves of $\FM$ are projectively dense. The last case occurs if, and only if, $D$ is not a square. 
\end{cor}

\subsection{in $\mathcal{H}^{hyp}(g-1,g-1)$}

\begin{prop}\label{hyper}
Let $\mathcal{M}$ be a nonabsolute rank one affine invariant orbifold in $\mathcal{H}^{hyp}(g-1,g-1)$. Then $\M$ has property $\mathcal{P}$ if and only if it is nonarithmetic.
\end{prop}

\begin{proof}
Let $q$ be a translation surface in $\M$, and choose a periodic direction on $q$. Up to applying Rel, we can ensure that direction is $\FM$-stable (which is equivalent to $\mathcal{F}$-stable here as there are only two singularities). Note that the singularity present on a boundary component of a cylinder is different from the singularity present on the other boundary component. This is due to the fact that the hyperelliptic involution exchanges the singularities of $\omega$ while fixing all the cylinders (see Lemma 2.2 in \cite{lindsey2015counting}). A consequence of that is that any $u \in \K$ has the form $u=(\delta_1tc_1^{-1},\cdots,\delta_mtc_m^{-1})$ with the $\delta_i$ belonging to $\{1,-1\}$. If $\K$ is rational, all the entries of $u$ are pairwise commensurable, a contradiction with the fact that $\M$ is nonarithmetic by \ref{field}. Therefore $\M$ has property $\mathcal{P}$ if and only if $\M$ is nonarithmetic.  
\end{proof}

\begin{thm}
Let $\M$ be a nonabsolute rank 1 affine invariant orbifold contained in $\mathcal{H}(g-1^2)^{hyp}$, and let $q$ be a surface in $\M$. Then the leaf of $\FM_q$ is either closed or projectively dense in $\M$. The latter case occurs if, and only if $\M$ is nonarithmetic. 
\end{thm}

Paul Apisa has classified the rank one affine invariant orbifolds in the hyperelliptic strata and showed that if $\M$ is a nonarithmetic rank $1$ affine invariant orbifold, then it is a translation cover of a surface in $\Omega E_D(1,1)$. See \cite{apisa2017rank}.

\section{Dense isoperiodic leaves in $\mathcal{H}(2,1,1)$}\label{full}

This section is dedicated to proving Theorem \ref{C}. 

\begin{thm}[Theorem C]
Let $q \in \Omega E_D(2,1,1)$ where $D$ is not a square. The leaf $\mathcal{F}_q$ of $q$ is projectively dense in $\mathcal{H}(2,1,1)$. 
\end{thm}

\begin{proof}
It follows from the proof of Theorem \ref{A} and Proposition \ref{pforeigen} that the closure of $\mathcal{F}_q$ is $\GL$-invariant. By \cite{eskin2015isolation}, it is then an affine invariant orbifold $\M$ and we denote by $r$ its rank. As in the proof of Proposition \ref{pforeigen}, it can be assumed that $q$ is the following surface: 

\begin{center}
\begin{tikzpicture}

\draw (0,0) rectangle (1,.5); 
\draw (0,.5) rectangle (2,1); 
\draw (1,1) rectangle (2,1.5); 
\draw (1,1.5) rectangle (3,2);
\draw (2,2) rectangle (3,2.5); 

\draw (0,0) node{$\bullet$}; 
\draw (1,0) node{$\bullet$};
\draw (0,1) node{$\bullet$}; 
\draw (1,1) node{$\bullet$};
\draw (2,1) node{$\bullet$}; 

\draw (1,1.5) node{$\times$}; 
\draw (2,1.5) node{$\times$};
\draw (3,1.5) node{$\times$};
\draw (2,2.5) node{$\times$};
\draw (3,2.5) node{$\times$};

\draw (0,.5) node{$\otimes$};
\draw (1,.5) node{$\otimes$};
\draw (2,.5) node{$\otimes$};
\draw (1,2) node{$\otimes$};
\draw (2,2) node{$\otimes$};
\draw (3,2) node{$\otimes$};

\draw (.5,-.25) node[scale = .7]{A};
\draw (.5,1.25) node[scale = .7]{A};
\draw (2.5,1.25) node[scale = .7]{B};
\draw (2.5,2.75) node[scale = .7]{B};

\draw (.5,.25) node{1};
\draw (1,.75) node{2};
\draw (1.5,1.25) node{3};
\draw (2,1.75) node{4};
\draw (2.5,2.25) node{5};

\end{tikzpicture}
\end{center}

Notice that $\frakK = \mathrm{ker} \rho$, which is the same as saying that $\M$ is saturated by the isoperiodic foliation $\mathcal{F}$. Let $\nu_1 = I(\gamma_1,\cdot) \in \mathrm{ker}\rho_q$ where $\gamma_1$ is a homotopically trivial loop that circles clockwise the singularity $\bullet$. Let also $\nu_2 = I(\gamma_2,\cdot) \in \mathrm{ker}\rho_q$ where $\gamma_2$ is a homotopically trivial loop that circles clockwise the singularity $\times$. The vectors $w_{\nu_1}$ and $w_{\nu_2}$ belong to $\K$ and we compute that $w_{\nu_1} = (-c^{-1}_1,c^{-1}_2,-c^{-1}_3,0,0)$ and $w_{\nu_2} = (0,0, c^{-1}_3,-c^{-1}_4, c^{-1}_5)$. Let $V_1$ be the smallest rational subspace that contains $w_{\nu_1}$. The dimension of $V_1$ is at least 2 as otherwise $c_1$, $c_2$ and $c_3$ would be commensurable and since the Prym involution swaps the cylinder 1 with 5 and 2 with 4 we would get that all the cylinders are pairwise commensurable, a contradiction with $D$ not being a square. If $V_1 \subset \K$, then by dimension count we have $V_1 = \K$. In particular, any linear equation with rational coefficients satisfied by the entries of $w_{\nu_1}$ is also satisfied by that of $w_{\nu_2}$. But ${w_{\nu_1}}_{5}=0$ is such an equation that is not satisfied by the entries of $w_{\nu_2}$. This proves that $V_1$ is not contained in $\K$. If we assume that $r=1$ then by Proposition \ref{VK} we know that $\V$ is rational and thus $V_1 \subset \V$. Since $w_{\nu_1}$ and $w_{\nu_2}$ form a basis of $\K$, there is $v_1 \in V_1$ and a $t \in \mathbb{R}-\{0\}$ such that $\mu = v_1 + t \cdot w_{\nu_2}$. Any linear equation with rational coefficients satisfied by the entries of $w_{\nu_1}$ has to be satisfied by that of $v_1$. It comes that $v_{1,4} = v_{1,5} = 0$ and thus $\mu_4 = t\cdot {w_{\nu_2}}_{4}$ and $\mu_5 = t\cdot {w_{\nu_1}}_5$. This means that $h_4 = -h_5$, which is a contradiction as the height of cylinders is a positive number. This shows that the rank of $\M$ is greater than $1$. However, the classification of Aulicino and Nguyen of rank 2 affine invariant orbifolds (\cite{aulicino2016rank}) in genus 3 shows none of them is saturated by $\mathcal{F}$. Thus $r=3$, and by \cite{mirzakhani2018full}, it has to be the whole stratum, since the hyperelliptic locus is not saturated by $\mathcal{F}$ either.  
\end{proof}

\bibliographystyle{alpha}
\bibliography{bibdeflo.bib}

\end{document}